\documentclass[11pt]{amsart}
\usepackage[T1]{fontenc}
\usepackage{lmodern,amsmath,amssymb,amsthm,microtype}
\usepackage{needspace}
\usepackage[margin=1in]{geometry}
\usepackage[hidelinks]{hyperref}
\hypersetup{pdftitle={Linearity of graph-manifold groups}}
\newtheorem{theorem}{Theorem}
\newtheorem{lemma}[theorem]{Lemma}
\newtheorem{corollary}[theorem]{Corollary}
\newtheorem{remark}[theorem]{Remark}
\DeclareMathOperator{\GL}{\mathsf{GL}}
\DeclareMathOperator{\M}{\mathsf{M}}
\DeclareMathOperator{\SL}{\mathsf{SL}}

\newcommand{\R}{\mathbb R}

\newcommand{\Q}{\mathbb Q}
\newcommand{\Z}{\mathbb Z}
\newcommand{\K}{\mathbb K}

\title{Linearity of graph-manifold groups}
\author[D.~T.~Wise]{Daniel T. Wise}
\email{daniel.wise@weizmann.ac.il}
\address{Dept. of Mathematics\\
Weizmann Institute of Science}

\dedicatory{to Peter Shalen on his 80th}
\date{\today}
\begin{document}

\begin{abstract}
We complete the proof that every compact $3$-manifold group is linear. Specifically, we construct a faithful representation in $\GL_5(\Q[u])$ for a
finite-index subgroup of every closed
non-geometric graph-manifold group.
\end{abstract}

\thanks{The author used AI to extract Lemma~\ref{lem:degree},
which simplified the proof of Theorem~\ref{thm:product}, to generate the example in Section~\ref{sec:example}, and to polish the exposition.
The author takes responsibility for the text.}

\maketitle

\vspace{-1cm}
\section{Introduction}
A \emph{graph manifold} is a compact, connected,
irreducible $3$--manifold whose JSJ pieces are Seifert fibered.
We prove that every closed graph-manifold group is $\R$-linear.
Together with the known geometric and mixed cases, this resolves the
linearity problem for compact $3$-manifold groups
\cite[prob~3.37(a)]{K3}. Thurston's original question asks for an
embedding in $\GL_4(\R)$ \cite[prob~3.33]{Kirby} but Button \cite{Button} exhibited
closed graph-manifold groups that do not embed in $\GL_4(\K)$ for any field $\K$. Here the dimension depends on the manifold.

Linearity was already known for the $3$-manifolds carrying one of
Thurston's homogeneous geometries. Przytycki-Wise proved virtual
specialness, and hence $\Z$-linearity, for mixed $3$-manifold groups
\cite[thm~1.1 and cor~1.2]{PWmixed} and for graph manifolds
with nonempty boundary \cite{PWboundary}. Liu proved that an aspherical
compact graph manifold has virtually special $\pi_1$ exactly
when it admits a nonpositively curved metric
\cite[thm~1.1]{Liu}. The remaining case was therefore that of
closed non-geometric graph manifolds without such a metric.

%
\begin{theorem}\label{thm:product}
Let $M$ be a closed connected $3$-manifold decomposed along tori into
finitely many product pieces $\Sigma_v\times S^1$. Assume each
$\Sigma_v$ is a compact connected orientable surface of positive genus
with a positive even number of boundary components, and the two
fiber slopes are distinct on every gluing torus. Then there is a
faithful representation
$\pi_1M\longrightarrow\GL_5(\Q[u])$
where $u$ is an indeterminate. In particular, $\pi_1M$ embeds in
$\GL_5(\R)$.
\end{theorem}

We prove Theorem~\ref{thm:product} in
Sections~\ref{sec:preliminaries}-\ref{sec:faithfulness}.
In Section~\ref{sec:manifolds} we give the finite-cover reduction
and deduce linearity for all compact $3$-manifolds.
The construction has three steps. Surface holonomy and a Heisenberg
summand give faithful vertex representations in dimension $2+3$.
On each boundary torus, the two fiber directions act on separate
2-dimensional summands, and exchanging these summands matches the gluing.
Finally, a polynomial shear makes each edge traversal in a reduced word
increase the degree by one, proving faithfulness.
We work through an example in Section~\ref{sec:example}.

The faithfulness argument uses Shalen’s leading-coefficient method
\cite[prop~1.3]{Shalen}. This work was motivated by our application of Shalen's method to  tubular groups
\cite{Wise_Linear_Tubular}. Douba-Tsouvalas give related combination
criteria using a polynomial unipotent shear
\cite[thm~2.1]{DT} and block matrices
\cite[thm~6]{DTtrans}. 
Button explored Shalen's approach
for graph manifolds and identified the difficulty of matching diagonal
peripheral representations in the closed case \cite[\S7]{Button}.

Douba proved that if a closed orientable irreducible non-Seifert
graph-manifold group admits a faithful $\R$-linear representation
without nonidentity unipotents, then the manifold admits a
nonpositively curved metric \cite[thm~1.3]{DoubaCAT}.
His earlier result for graph manifolds with one JSJ torus and
product pieces \cite[thm~1.8]{DoubaVU} uses an argument similar
to Button's for Gersten's group \cite[thm~4.5]{ButtonUnipotent}.

\section{Surface holonomy}\label{sec:preliminaries}

\begin{lemma}[Surface holonomy]\label{lem:holonomy}
Let $\Sigma$ be a compact connected orientable surface of genus $g\ge1$
with an even number $m\ge2$ of boundary components, represented by
boundary loops $c_1,\ldots,c_m$ in $F=\pi_1\Sigma$.
There is a faithful representation
$j:F\to\SL_2(\Z)$ with each $j(c_i)$ nontrivial unipotent.
\end{lemma}

\begin{proof}
Let $T$ be a torus with one boundary component and
$\pi_1T=\langle a,b\rangle$.
The standard faithful modular-torus representation $j_0$ is given by
\[
j_0(a)=A=\begin{pmatrix}2&1\\1&1\end{pmatrix}
\ \text{ and } \
j_0(b)=B=\begin{pmatrix}2&-1\\-1&1\end{pmatrix}
\ \text{ so } \ \ \  
[A,B]=-\begin{pmatrix}1&6\\0&1\end{pmatrix}.
\]

There is a finite covering $\Sigma\to T$, and pulling back $j_0$
gives a faithful representation $\rho$.
A boundary image has trace $-2$ precisely when its covering degree
is odd. Since the total covering degree is $2g+m-2$, which is even,
an even number of boundary images have trace $-2$.

Using a standard surface presentation, multiply these boundary
matrices by $-I$, leaving the handle images unchanged.
The signs cancel in the surface relation, so this defines a
representation $j$. And it remains faithful since if $j(f)=I$, then
$\rho(f)^2=I$, hence $f^2=1$, and therefore $f=1$ since $F$ is
torsion-free. Every boundary image is now nontrivial unipotent.
\end{proof}

Choose nonzero rational vectors $\ell_i\in\ker(j(c_i)-I)$.
The boundary subgroups $\langle c_i\rangle$ are maximal cyclic
and pairwise nonconjugate in $F$.
If $\det(\ell_i,j(f)\ell_k)=0$, the unipotents
$j(c_i)$ and $j(fc_kf^{-1})$ have the same fixed line and hence
commute. Faithfulness and cyclic centralizers in $F$ give
$f\langle c_k\rangle f^{-1}=\langle c_i\rangle$.
Pairwise nonconjugacy and malnormality therefore imply
$i=k$ and $f\in\langle c_i\rangle$.
The converse is immediate, so
\begin{equation}\label{eq:detection}
\det(\ell_i,j(f)\ell_k)=0
\quad\Longleftrightarrow\quad
i=k\ \text{and}\ f\in\langle c_i\rangle.
\end{equation}

\section{Matching the peripheral representations}

Assume the hypotheses of Theorem~\ref{thm:product}.
We first construct the vertex representations and gluing matrices implementing
the edge identifications.
In the next section we deform the gluing matrices and prove faithfulness.

The decomposition of $M$ gives a graph of groups with vertex groups
$H_v=F_v\times\langle z_v\rangle$, where $F_v=\pi_1\Sigma_v$.
For an oriented edge $e:v\to w$, let $\bar e$ denote its reverse.
Let $G_e=G_{\bar e}$ be the torus group $\Z^2$, with inclusion
$\iota_e:G_e\to H_v$ and image
\[
P_e=\iota_e(G_e)=\langle c_e,z_v\rangle.
\]
Here the $c_e$ are the corresponding boundary loops, chosen from
standard surface presentations of the $F_v$.
We represent elements of $\pi_1M$ by loop words at a fixed base
vertex, with vertex-group elements inserted along the path.
The relations are the vertex-group relations and
\[
e\bar e=1,\qquad \iota_e(g)e=e\iota_{\bar e}(g)\quad(g\in G_e).
\]
A \emph{pinch} is a subword $e h\bar e$ with $h\in P_{\bar e}$.

Transporting the neighboring fiber $z_w$ across the gluing gives
$c_e^{p_e}z_v^{q_e}$. The fiber slopes are distinct, so $p_e\ne0$.
Let $r_e=-q_e/p_e$.
For $c_e^s z_v^t\in P_e$, the coefficients $s$ and $r_es+t$ vanish
on the local and neighboring fiber, respectively. We use them to
make each fiber act trivially on one summand and nontrivially on the other.

\subsection{The vertex representation}

For each surface choose $j_v$ and cusp vectors $\ell_e$ as in
Lemma~\ref{lem:holonomy}. Let $U_v=\Q^2$ and $A_e=j_v(c_e)-I$.
Thus $A_e^2=0$.

At a fixed vertex, suppress the vertex subscript and enumerate the
boundary loops as $c_1,\ldots,c_m$, with assigned values
$r_1,\ldots,r_m$. Using $[a,b]=aba^{-1}b^{-1}$ we have the usual presentation:
\[
F=\langle a_1,b_1,\ldots,a_g,b_g,c_1,\ldots,c_m
\ \mid \  [a_1,b_1]\cdots[a_g,b_g]\cdot c_1\cdots c_m=1 \rangle.
\]
%
Let $W=\Q^3$ have basis $w_1,w_2,w_3$. Let $E_{ij}$ be the matrix
whose only nonzero entry is a $1$ in position $(i,j)$.
Define a Heisenberg representation $\eta:F\to\GL(W)$ by
\[
\eta(c_i)=I+r_iE_{13},\qquad
\eta(a_1)=I-\Bigl(\sum_i r_i\Bigr)E_{12},\qquad
\eta(b_1)=I+E_{23},
\]
with the other handle generators trivial. The identity
$[I+sE_{12},I+E_{23}]=I+sE_{13}$ verifies the surface relation:
one handle absorbs the sum of the boundary values.
The image of $\eta$ commutes with $E_{13}$ and fixes the column $w_1$ and
the row $w_3^*$, which reads the third coordinate.

On $V_v=U_v\oplus W$ define
\[
\rho_v(fz_v^n)=j_v(f)\oplus\bigl(\eta_v(f)(I+nE_{13})\bigr).
\]
This is a representation because $\eta_v(F_v)$ commutes with $E_{13}$ and
$E_{13}^2=0$. It is faithful as the first summand detects $f$, and then the
second detects $n$. Its peripheral restrictions are
\[
\rho_v(c_e^s z_v^t)
=(I+sA_e)\oplus\bigl(I+(r_es+t)E_{13}\bigr).
\]

\subsection{The gluing maps}

The local fiber acts on $W$ by $I+E_{13}$ and fixes $U_v$, and the neighboring
fiber acts on $U_v$ by $I+p_eA_e$ and fixes $W$.
We choose a basis of $U_v$ that matches its unipotent block to the
block on $\langle w_1,w_3\rangle$.

Regard vectors as columns and covectors as rows: $x\phi$ is the
matrix of the map $y\mapsto\phi(y)x$, whereas $\phi x$ is a scalar.
Since $\operatorname{im}A_e=\ker A_e=\Q\ell_e$, there is a unique
nonzero row $\beta_e\in U_v^*$ with
\[
p_eA_e=\ell_e\beta_e.
\]
It satisfies $\beta_e\ell_e=0$, and both $\ell_e$ and $\beta_e$
are fixed by $j_v(c_e)$. Choose $\ell'_e\in U_v$ with $\beta_e\ell'_e=1$.
Then $(I+p_eA_e)\ell'_e=\ell'_e+\ell_e$, just as
$(I+E_{13})w_3=w_3+w_1$.

For $e:v\to w$, define $C_e:V_w\to V_v$ by its action on a basis:
\[
\begin{array}{c|ccccc}
\text{vector in }V_w&\ell_{\bar e}&\ell'_{\bar e}&w_1&w_3&w_2\\
\hline
\text{image in }V_v&w_1&w_3&\ell_e&\ell'_e&w_2
\end{array}
\]
Thus $C_e$ exchanges the two unipotent blocks and fixes the remaining
coordinate. It satisfies $C_{\bar e}=C_e^{-1}$ and matches both fiber
actions: each fiber is local on one side and neighboring on the other.

The two fibers generate a finite-index subgroup of $G_e$.
Every peripheral matrix is $I+D$ with $D^2=0$. 
Since $(I+D)^k=I+kD$, such matrices are determined by any
fixed positive power.
Compatibility on the two fibers therefore implies
compatibility on the whole torus group:
\begin{equation}\label{eq:intertwiner}
\rho_v(\iota_e(a))C_e
=C_e\rho_w(\iota_{\bar e}(a))\qquad(a\in G_e).
\end{equation}

\section{Distinguishing reduced words}\label{sec:faithfulness}

The degree of a nonzero polynomial matrix is the largest degree
of its entries. Write $[u^k]A$ for the coefficient matrix of $u^k$.
When $\deg A=k$, this is its \emph{leading coefficient}.

\begin{lemma}[Degree growth]\label{lem:degree}
Let $A,B\in\M_d(\Q[u])$ have degrees $k$ and $1$, with leading
coefficients $a\phi$ and $b\psi$, respectively, where $a,b$ are
nonzero columns and $\phi,\psi$ are nonzero rows.
If $M\in\M_d(\Q)$ satisfies $\phi Mb\ne0$, then $AMB$ has
degree $k+1$ and leading coefficient $a(\phi Mb)\psi$.
\end{lemma}
\begin{proof}
Only the leading terms contribute to degree $k+1$, and
\[
[u^{k+1}](AMB)
=(a\phi)M(b\psi)=a(\phi Mb)\psi\ne0.
\qedhere
\]
\end{proof}

The new leading coefficient is again rank one, so Lemma~\ref{lem:degree} can be
applied successively.

\begin{proof}[Proof of Theorem~\ref{thm:product}]
We deform each gluing matrix by a shear that commutes with the
peripheral group.
For an edge $e$ starting at $v$, define a column
$x_e\in V_v$ and a row $\phi_e\in V_v^*$ by
\[
x_e=\ell_e+w_1,\qquad \phi_e=\beta_e+w_3^*,
\]
using the direct sum $V_v=U_v\oplus W$. Both are fixed by $\rho_v(P_e)$.
Taking sums makes them compatible with the exchange of the two blocks:
\begin{equation}\label{eq:transport}
\phi_e x_e=0,\qquad C_e x_{\bar e}=x_e,\qquad \phi_e C_e=\phi_{\bar e}.
\end{equation}
For edges $e,d$ starting at $v$ and $h=fz_v^n\in H_v$, the Heisenberg
summand fixes $w_1$ and $w_3^*w_1=0$, so
\[
\phi_e\rho_v(h)x_d=\beta_ej_v(f)\ell_d.
\]
Since $\ker\beta_e=\Q\ell_e$, 
Equation~\eqref{eq:detection} gives
\begin{equation}\label{eq:transition}
\phi_e\rho_v(h)x_d=0
\quad\Longleftrightarrow\quad e=d\ \text{and}\ h\in P_e.
\end{equation}
For a subword $e h d$ passing through $v$, the relevant pairing is
$\phi_{\bar e}\rho_v(h)x_d$, and it vanishes exactly when $e h d$ is a pinch.

Choose signs $\varepsilon_e\in\{\pm1\}$ with
$\varepsilon_{\bar e}=-\varepsilon_e$, and let
\[
X_e \ = \ (I+\varepsilon_eu x_e\phi_e)C_e
   \ = \ C_e+\varepsilon_eu x_e\phi_{\bar e}.
\]
The shear commutes with $\rho_v(P_e)$ and has inverse
$I-\varepsilon_eu x_e\phi_e$, since $\phi_e x_e=0$.
Equations \eqref{eq:transport} and \eqref{eq:intertwiner} therefore give
\[
X_{\bar e}=X_e^{-1} \ \text{ and } \ 
\rho_v(\iota_e(g))X_e=X_e\rho_w(\iota_{\bar e}(g))
\ \text{ for } \ g\in G_e.
\]
Thus the vertex and edge matrices respect the path relations. Their
products along loops at a fixed base vertex define
$\rho:\pi_1M\to\GL_5(\Q[u])$.

Let $g\ne1$ and choose a based path word with the fewest edge traversals:
\[
g=h_0e_1h_1\cdots e_nh_n,\qquad
e_i:v_{i-1}\to v_i,\quad h_i\in H_{v_i},\quad v_n=v_0.
\]
If $n=0$, vertex-group faithfulness gives $\rho(g)\ne I$.
Suppose $n>0$. Minimality excludes pinches, so \eqref{eq:transition} gives
\[
\phi_{\bar e_i}\rho_{v_i}(h_i)x_{e_{i+1}}\ne0\qquad(1\le i<n).
\]
Omit the invertible constant matrices at the two ends, which does not
change degree. We prove inductively that the product through $X_{e_i}$
has degree $i$ and leading coefficient a nonzero scalar multiple of
$x_{e_1}\phi_{\bar e_i}$.
The case $i=1$ follows from the formula for $X_{e_1}$.
The induction step follows from Lemma~\ref{lem:degree} and the
nonzero pairings above.
Thus $\deg\rho(g)=n>0$, so $\rho(g)\ne I$.

Evaluation at a real number transcendental over $\Q$
embeds $\Q[u]$ in $\R$ and preserves faithfulness.
\end{proof}

%
\section{Finite covers and consequences}\label{sec:manifolds}

\begin{lemma}[Finite-cover reduction]\label{lem:cover}
Every closed connected irreducible non-geometric graph
manifold has a finite cover satisfying the hypotheses of
Theorem~\ref{thm:product}.
\end{lemma}
\begin{proof}
Kapovich-Leeb \cite[lem~2.1]{KL} give a finite cover whose Seifert
pieces are products $\Sigma_v\times S^1$, with each $\Sigma_v$
orientable of genus at least two and with nonempty boundary. The incident fiber slopes are distinct.
It remains to arrange that each $\Sigma_v$ has an even number of
boundary components.

Map each vertex group to $\Z/2$ by sending one handle generator to $1$
and all other handle generators, boundary generators, and the fiber
to $0$. Send the stable letters of a maximal-tree presentation to $0$.
The relations hold and the edge groups map trivially, so this induces a homomorphism
restricting to a surjection on each vertex group.
In its connected double cover, each piece lifts to a connected double
cover of its base times $S^1$, with two degree-one lifts of each
boundary circle. A base of genus $g$ with $m$ boundary circles lifts
to one of genus $2g-1$ with $2m$ boundary circles. Thus every piece
satisfies the required hypotheses, and the fiber slopes remain distinct.
\end{proof}

We use two closure properties. For $R=\Q[u]$ or $\R$, a faithful
representation $P\to\GL_m(R)$ induces a faithful representation
$G\to\GL_{md}(R)$ whenever $[G:P]=d<\infty$. A kernel element
fixes every coset, hence lies in $P$. Its action on the $P$-summand
is the given faithful representation, so the element is trivial.

Finite free products of f.g.\ $\R$-linear groups are
$\R$-linear. Put the entries of the generating matrices in a 
finitely generated field $\K\subset\R$. Its algebraic closure has infinite
degree over $\K$, so Marciniak's theorem \cite{Marciniak} gives a faithful representation of a free
product over $\K(t)$. Specializing $t$ to a real number transcendental
over $\K$ and iterating proves the assertion.

\begin{corollary}[Main result]\label{cor:graph}
The fundamental group of every closed graph manifold is $\R$-linear.
\end{corollary}
\begin{proof}
The geometric case is already known.
Otherwise, Lemma~\ref{lem:cover} and
Theorem~\ref{thm:product} give a finite cover with
$\R$-linear fundamental group.
Linearity passes to finite extensions.
\end{proof}

\begin{corollary}\label{cor:all}
The fundamental group of every compact connected $3$-manifold is
$\R$-linear.
\end{corollary}
\begin{proof}
A compact manifold with nonempty boundary is a retract of its
closed double. By finite-extension closure, it therefore suffices
to consider closed orientable manifolds.
Prime decomposition and free-product closure reduce further to
irreducible manifolds, since $S^2\times S^1$ has cyclic fundamental
group.
The geometric cases are already known.
The remaining cases are graph manifolds or mixed.
Corollary~\ref{cor:graph} handles the former, and the latter have
$\Z$-linear fundamental groups by \cite[cor~1.2]{PWmixed}.
\end{proof}

\begin{remark}\label{rem:SLZ}
As in \cite{Wise_Linear_Tubular}, with a bit more effort we can embed in $SL_n(\Z[u])$.
\end{remark}

\section{An example}
\label{sec:example}

Let $M$ be obtained by gluing the two boundary tori of
$\Sigma_{1,2}\times S^1$, where $\Sigma_{1,2}$ is a genus-one surface
with two boundary components. Its vertex group and gluing relations are
\[
\begin{aligned}
H&=\langle a,b,c_1,c_2,z\mid
 [a,b]c_1c_2=1,\ [a,z]=[b,z]=[c_1,z]=[c_2,z]=1\rangle,\\
G=\pi_1M&=\langle H,t\mid
 t c_2t^{-1}=c_1^{-1},\quad tzt^{-1}=c_1z\rangle.
\end{aligned}
\]
Equivalently, letting $c=c_1$ and $d=t^{-1}$ gives
\[
G=\langle a,b,c,d,z\mid [a,b][c,d]=1,\
 zaz^{-1}=a,\ zbz^{-1}=b,\ zcz^{-1}=c,\ zdz^{-1}=dc\rangle.
\]
Thus $M$ is the mapping torus of a nonseparating Dehn twist on a closed
genus-two surface. In the notation of \cite{DoubaVU}, its gluing matrix is
$\left(\begin{smallmatrix}1&1\\0&-1\end{smallmatrix}\right)$.
The criterion in \cite[thm~2.2]{DoubaVU} shows $M$ is not nonpositively curved,
so $G$ is not virtually special by \cite[thm~1.1]{Liu}.

Here the parameters at both ends of the edge are $p=q=1$ and $r=-1$.
For the surface matrices displayed below, choose
\[
\begin{aligned}
\ell_1&=(6,0)^{\mathsf t},&
\ell'_1&=(0,-1)^{\mathsf t},\\
\ell_2&=(0,6)^{\mathsf t},&
\ell'_2&=(1,0)^{\mathsf t}.
\end{aligned}
\]
Taking $\varepsilon_t=1$, the construction gives the following faithful
representation $\rho:G\to\SL_5(\Q[u])$.

Let $A=\rho(a)$, $B=\rho(b)$, $C_i=\rho(c_i)$,
$Z=\rho(z)$, and $T=\rho(t)$, and let $E_{ij}$ denote the usual
matrix units in $\M_5(\Z)$. Then
\[
\begin{gathered}
\begin{aligned}
A&=\begin{pmatrix}-1&-3\\3&8\end{pmatrix}
   \oplus\begin{pmatrix}1&2&0\\0&1&0\\0&0&1\end{pmatrix},
&
B&=\begin{pmatrix}4&5\\-1&-1\end{pmatrix}
   \oplus\begin{pmatrix}1&0&0\\0&1&1\\0&0&1\end{pmatrix},
\end{aligned}\\[6pt]
\begin{aligned}
C_1&=I_5-6E_{12}-E_{35},\\[4pt]
C_2&=I_5+6E_{21}-E_{35},\\[4pt]
Z&=I_5+E_{35},
\end{aligned}
\qquad
T=\begin{pmatrix}
6u&0&6&0&6u\\
0&0&0&0&-1\\
u&1/6&0&0&u\\
0&0&0&1&0\\
1&0&0&0&0
\end{pmatrix}.
\end{gathered}
\]

To illustrate the degree growth, consider the pinchless word
\[
w=[tat^{-1},b]=tat^{-1}bt a^{-1}t^{-1}b^{-1}.
\]
Introduce the columns and rows
\[
\begin{aligned}
x&=(6,0,1,0,0)^{\mathsf t},&
y&=(0,6,1,0,0)^{\mathsf t},\\
\alpha&=(1,0,0,0,1),&
\beta&=(0,-1,0,0,1).
\end{aligned}
\]
Then
\[
[u]T=x\alpha,\qquad [u]T^{-1}=-y\beta.
\]
The three intervening pairings are
\[
\alpha Ay=-18,\qquad \beta Bx=6,\qquad \alpha A^{-1}y=18.
\]
Let $P_k$ be the matrix of the prefix ending at the $k$-th occurrence of
$t^{\pm1}$ in $w$. We then have:
\[
\begin{array}{c|cccc}
k&1&2&3&4\\\hline
\deg P_k&1&2&3&4\\
{[u^k]P_k}&x\alpha&18x\beta&108x\alpha&-1944x\beta
\end{array}
\]
Finally, $\rho(w)=P_4B^{-1}$ and
\[
[u^4]\rho(w)
\ = \ -1944x\beta B^{-1}
\ = \ 1944x(1,4,0,0,-1).
\]
Thus $\deg\rho(w)=4$, exactly its number of edge traversals.

\section*{Tool Use}
The author used AI tools to write Python code to
explore representations, to extract Lemma~\ref{lem:degree},
which simplified the proof of Theorem~\ref{thm:product},
to generate the example in Section~\ref{sec:example},
to double-check Remark~\ref{rem:SLZ},
to help identify additional references, to polish the exposition,
to computationally stress-test the proofs, and to help write
this declaration. The author takes full responsibility for
the mathematical arguments and final text.

\end{document}